\documentclass{amsart}      

\usepackage{amssymb,amsmath,graphicx,mathrsfs,mathabx,tikz,url}

\usepackage{srcltx}

\usetikzlibrary{matrix}

\newtoks\prt

\newtheorem{thm}{Theorem}[section]

\newtheorem{prop}[thm]{Proposition}

\newtheorem{fact}[thm]{Fact}

\theoremstyle{definition}

\newtheorem*{remark}{Remark}

\def\eqn#1$$#2$${\begin{equation}\label#1#2\end{equation}}

\let\tilde\widetilde

\def\C{\mathcal C}

\def\Dc{\mathscr D} 

\def\Dcp{\mathscr D^{\,\prime}}

\def\F{\mathcal F}

\def\H{\mathcal H}

\def\ep{\varepsilon}

\def\e{{\boldsymbol  e}}

\def\o{{\boldsymbol o}}

\def\a{{\boldsymbol  a}}

\def\g{\boldsymbol g}

\def\f{\boldsymbol f}

\def\x{\boldsymbol x}

\def\y{\boldsymbol y} 

\def\h{\boldsymbol h}

\def\en{\mathbb N} 

\def\er{\mathbb R}

\def\dist{\operatorname{dist}}

\def \div {\operatorname{div}}

\def \Lip {\operatorname{Lip}}

\def\spt{\operatorname{spt}}

\def \reg {\partial _{\kern1pt\text{reg}}}

\newcommand{\ip}[2]{\left\langle#1,#2\right\rangle}

\def\di{\,\mbox{\rm d}}

\newcommand{\norm}[1]{\left\|#1\right\|}

\newcommand{\esssup}{\operatorname{ess sup}}

\newcommand{\abs}[1]{\left|#1\right|}

\newcommand{\setsep}{;\,}

\newcommand\ev[2]{\langle #1,#2\rangle}

\begin{document}

\title[Lipschitz-free spaces over convex domains in finite-dimensional spaces]{Isometric representation of Lipschitz-free spaces over convex domains in finite-dimensional spaces}

\author{Marek C\'uth, Ond\v{r}ej F.K. Kalenda and Petr Kaplick\'y}

\address{Department of Mathematical Analysis \\
Faculty of Mathematics and Physic\\ Charles University\\
Sokolovsk\'{a} 83, \\ 186 75, Praha 8, Czech Republic}

\email{cuth@karlin.mff.cuni.cz}

\email{kalenda@karlin.mff.cuni.cz}

\email{kaplicky@karlin.mff.cuni.cz}



\thanks{M.~C\'uth is a junior researcher in the University Center for Mathematical Modelling, Applied Analysis and Computational Mathematics (MathMAC). M.~C\'uth and O.~Kalenda were supported in part by the grant GA\v{C}R P201/12/0290. P. Kaplick\'y is a member of the Ne\v{c}as Center for Mathematical Modeling.}

\begin{abstract} 
Let $E$ be a finite-dimensional normed space and $\Omega$ a nonempty convex open set in $E$.
We show that the Lipschitz-free space of $\Omega$ is canonically isometric to the quotient of $L^1(\Omega,E)$ by the subspace consisting of vector fields with zero divergence in the sense of distributions on $E$.
\end{abstract}

\maketitle

\section{Introduction}

Given a metric space $M$, by $\F(M)$ we denote the \emph{Lipschitz-free space over~$M$}. This is a Banach space whose linear structure somehow reflects the metric structure of~$M$. The study of Banach space theoretical properties of Lipschitz-free spaces was initiated by a paper by G. Godefroy and N. Kalton \cite{gk03}, where the authors proved, using this notion, e.g. that if a separable Banach space~$Y$ is isometric (not necessarily linearly) to a subset of a Banach space~$X$, then $Y$ is already linearly isometric to a subspace of $X$. Soon after, the study of Lipschitz-free Banach spaces became an active field of study, see e.g. \cite{glz}, \cite[Section 10]{ost} and the results mentioned below. However, the structure of these spaces is still very poorly understood to this day. For example, it is not known whether $\F(\er^2)$ is isomorphic with $\F(\er^3)$.

Let us recall the construction of $\F(M)$. Choose a distinguished ``base point'' $0\in M$ and denote by $\Lip_0(M)$ the space of all real-valued Lipschitz functions on $M$ that map $0\in M$ to $0\in\er$. It becomes a Banach space if we define the norm of $f$ to be its minimal Lipschitz constant. For any $x\in M$ we denote by $\delta(x)\in\Lip_0(M)^*$ the evaluation functional, i.e. $\ev{\delta(x)}f=f(x)$ for $f\in\Lip_0(M)$.  It is easy to see that $\delta$ is an isometric embedding of $M$ into $\Lip_0(M)^*$. The space $\F(M)$ is defined to be the closed linear span of $\{\delta(x)\setsep x\in M\}$ with the dual space norm denoted simply by $\norm{\cdot}$. For details and additional properties see Section 2 or \cite[Section~1]{CDW}.

There is a large number of results on the structure of Lipschitz-free space (some of them are recalled below), but an explicit isometric representation of $\F(M)$ is known only for very special $M$. It is known that $\F(\er)$ is isometric to $L^1(\er)$ (cf. \cite[page 128]{gk03} or \cite[page 33]{glz}) and there is an isometric representation for certain discrete spaces. Our main result is an explicit isometric representation of $\F(\Omega)$ where $\Omega$ is a nonempty convex domain in a finite-dimensional normed space. It reads as follows:

\begin{thm}\label{t:main} Let $E$ be a real normed space of dimension $d\in\en$ and $\Omega\subset E$ be a nonempty convex open subset. Then the Lipschitz-free space $\F(\Omega)$ is canonically isometric to the quotient space
$$L^1(\Omega,E)/\{\g\in L^1(\Omega,E)\setsep\div\g=0\mbox{ in the sense of distributions on }\er^d\}.$$

Moreover, if  $\o\in\Omega$  is the base point and $\a\in\Omega$ is arbitrary, then in this identification we have
\begin{multline*}\delta(\a)=[\g]\mbox{ if and only if }\g\in  L^1(\Omega,E)\\\mbox{ and }\div\g=\ep_\o-\ep_\a\mbox{ in the sense of distributions on }\er^d,\end{multline*}
where $\ep_{\x}$ denotes the Dirac measure supported at $\x$.
\end{thm}

Let us explain the notation and terminology used in the theorem. First, by $[\g]$ we denote the equivalence class of $\g$ as an element of the quotient space. Further, since $L^1(\Omega,E)$ is canonically isometrically embedded into $L^1(\er^d,E)$ (just extend any $\g\in L^1(\Omega,E)$ by zero outside $\Omega$), any $\g\in L^1(\Omega,E)$ can be viewed either as a regular distribution on $\Omega$ or as a regular distribution on $\er^d$. Thus, $\div\g$ in the sense of distributions on $\er^d$ is the distributional divergence of the regular distribution on $\er^d$ induced by the described extension of $\g$. This is quite an important matter, since the result would become false if we considered divergence in the sense of distributions on $\Omega$.

The above theorem is new for $d\ge2$, but it covers also the known case $d=1$. Indeed, if $E=\er$ and $\Omega=(a,b)$ where $-\infty\le a<b\le\infty$, then $\F((a,b))$ is canonically identified with $L^1((a,b))$ and, if $0\in (a,b)$, then in this identification we have 
$$\delta(x)=\begin{cases} \chi_{(0,x)},& x\in(0,b),\\ 0, & x=0,\\ -\chi_{(-x,0)},& x\in(a,0).\end{cases}$$
This case is covered by our theorem, since in dimension one the divergence is just the derivative and the only $f\in L^1(a,b)$, whose derivative in the sense of distributions on $\er$ is zero, is the constant zero function.
(We point out that we work in distributions on $\er$, not in distributions on $(a,b)$. If $(a,b)$ is a bounded interval, then constant functions on $(a,b)$ belong to $L^1(a,b)$ and their derivative in the sense of distributions on $(a,b)$ is zero unlike their derivative in the sense of distributions on $\er$.)
Moreover, it is clear that the distributional derivative of the characteristic function $\chi_{(u,v)}$ is $\ep_u-\ep_v$.

Our result was motivated in part by a result of N.~Lerner \cite{lerner} who proved that $\Lip_0(\er^d)$ is isomorphic to the dual of the quotient described in our theorem. This is closely related to our result, since $\F(M)^*$ is isometric to $\Lip_0(M)$ for any metric space $M$ (see the next section). However, it was not known to us whether $\F(M)$ is the unique predual of $\Lip_0(M)$, therefore to give a representation of $\F(\er^d)$ we described the mapping $\delta$. Let us note that it has been proved recently by N. Weaver \cite{weaver} that $\F(M)$ actually is the unique predual of $\Lip_0(M)$ if $M$ is a convex set in a Banach space; however, the description of the mapping $\delta$ could be of an independent interest. We further extended the results for other norms on $\er^d$ (which is easy) and to general convex domains (which requires some additional work).

Let us note that during the review process of this paper we found out that the problem of characterizing $\F(\er^d)$ has been independently investigated in the Master thesis \cite{flores}, where some partial results were obtained, and that
%, where the author came to a different description of $\Lip_0(\er^d)$ and tried to describe $\F(\er^d)$ using this different approach. As far as we know, his independent work is still in progress. Finally,
similar ideas as we use in our paper were also used in \cite[Appendix: The Sobolev space $W^{-1,1}$]{maly}, where the author describes the Sobolev space $W^{-1,1}$.

Let us summarize what is known to the authors about the structure of Lipschitz-free spaces over subsets of $\er^d$. If $X$ and $Y$ are Banach spaces, we write $X\equiv Y$, $X\cong Y$ and $X\hookrightarrow Y$ if $X$ is linearly isometric with $Y$, linearly isomorphic with $Y$ and isomorphic to a subspace of $Y$, respectively. If this is not the case, we write $X\nequiv Y$, $X\ncong Y$ and $X\not \hookrightarrow Y$.

\begin{fact}
\begin{enumerate}

	\item $\F(\er)\equiv L^1(\er)$.

    \item For any measure $\mu$, we have $\F(\er^2)\not \hookrightarrow L^1(\mu)$. Moreover, there exists a convergent sequence $K\subset\er^2$ such that $\F(K)\not\hookrightarrow L^1(\er)$

    \item If $M\subset\er^d$ has a nonempty interior, then $\F(M)\cong \F(\er^d)$.

    \item 	For every $M\subset\er^d$, the space $\F(M)$ is weakly sequentially complete. In particular, $c_0\not\hookrightarrow \F(M)$.

    \item $\F(\er^d)$ has a Schauder basis. Moreover, for every $M\subset\er^d$ bounded and convex, $\F(M)$ has the Schauder basis as well.

    \item For every $M\subset \er^d$, the space $\F(M)$ has the bounded approximation property (BAP). Moreover, if $M$ is compact and convex or $M = \er^d$, then $\F(M)$ has the metric approximation property (MAP) with respect to any norm on $\er^d$.

    \item If $K\subset \er^d$ is a countable compact set, then $\F(K)$ is a dual space which has the MAP. On the other hand, if $M\subset\er^d$ is convex and not reduced to a point, then $L^1([0,1])\hookrightarrow \F(M)$; in particular, $\F(M)$ is not a dual space.

    \item For a metric space $M$, the space $\F(M)$ is isometric to a subspace of an $L^1$ space if and only if $M$ isometrically embeds into an $\er$-tree.

\end{enumerate}
\end{fact}

		The assertion (1) is well known, see e.g. \cite[page 128]{gk03} or \cite[page 33]{glz}, the proof is easy using the above description and we recall it at the beginning of the following section. The first result in (2) was shown by A. Naor and G. Schechtmann \cite{naoSch}. As observed in \cite[Remark 4.2]{CDW}, this result actually follows already from \cite{kis} using minor modifications. The ``moreover'' part of (2) follows from \cite[Theorem 1.2 and Remark 4.5]{CDW}
The assertion (3) is proved in \cite[Corollary 3.5]{kauf} and the assertion (4) in  \cite[Theorem 1.3]{CDW}.
The assertion (5) is a result of E. Perneck\'a and P. H\'ajek, see Theorem 3.1 and  page 645 in \cite{hp}. 
The first statement of (6) follows from \cite[Proposition 2.3]{lp}, the case of $\F(\er^d)$  from \cite[Proposition 5.1]{gk03}, the remaining case is proved in \cite[Corollary 1.2]{ps}.
The first part of (7) is a special case of a result of A. Dalet \cite[Theorem 2.1]{aude}. On the other hand, if $M$ is a convex set, then $[0,1]$ bi-Lipschitz embeds into $M$; hence, $L^1([0,1])\equiv\F([0,1])\hookrightarrow \F(M)$. Therefore, if $\F(M)$ is a separable space which fails the Radon-Nikod\'ym property, hence it is not a dual space.
The assertion (8) is proved in \cite[Theorem 4.2]{godard}

It seems that the main open problems are whether $\F(\ell_1)$ is complemented in its bidual and whether $\F(M)$ has BAP for every uniformly discrete metric space, see \cite[Problem 16 and Problem 18]{glz}.

\section{Preliminaries}

In this section we introduce some notation we need to formulate and prove our results. Let us recall some basic facts concerning the Lipschitz-free spaces (for the proofs we refer to~\cite[Section~1]{CDW}). Let $(M,\rho,0)$ be a pointed metric space, i.e. a metric space with a distinguished ``base point'' denoted by $0$. The space $\F(M)$ described in the introduction is uniquely characterized by the following universal property:

Let $X$ be a Banach space and suppose that $L\colon M\to X$ is a Lipschitz mapping satisfying $L(0)=0$. Then there exists a unique bounded linear operator $\widehat{L}\colon\F(M)\to X$ extending $L$, i.e. the following diagram commutes:

\begin{center}\begin{tikzpicture}

  \matrix (m) [matrix of math nodes,row sep=3em,column sep=4em,minimum width=2em]
  {
    M & X \\
   \F(M) & X \\};

  \path[-stealth]

    (m-1-1) edge node [left] {$\delta_M$} (m-2-1)

            edge node [above] {$L$} (m-1-2)

    (m-2-1.east|-m-2-2) edge [dashed] node [above] {$\widehat{L}$} (m-2-2)

    (m-1-2) edge node [right] {$\mathrm{id}_X$} (m-2-2);    

\end{tikzpicture}\end{center}
Using this universal property of $\F(M)$ for $X=\er$ it can be rather easily shown that $\F(M)^*$ is linearly isometric to $\Lip_0(M)$.

The notation and terminology we use are relatively standard. For a Banach space $X$, $x\in X$ and $x^*\in X^*$ we denote by $\ip{x^*}{x}$ or by $\ip{x}{x^*}$ the application of the functional $x^*$ on the vector $x$; that is, $\ip{x^*}{x} = \ip{x}{x^*} = x^*(x)$.

Our basic setting is the following:

Let $E$ be a finite-dimensional real normed space with a fixed basis $\e_1,\dots,\e_d$, where $d\ge 2$. $E^*$ be the dual space to $E$ and $\e^*_1,\dots\e^*_d$ be the dual basis. Using the coordinates with respect to these bases we can canonically identify both $E$ and $E^*$ with the space $\er^d$ equipped with the corresponding norms. On $\er^d$ we will consider the standard Lebesgue $d$-dimensional measure $\lambda^d$ and the standard $(d-1)$-dimensional Hausdorff measure $\H^{d-1}$. These measures are transferred to $E$ and $E^*$ via the mentioned identification.

If $U\subset \er^d$ (or $U\subset E$) is a nonempty open set, by $\Dc(U)$ we denote the space of real-valued $\C^\infty$-smooth functions with a compact support in $U$. The symbol $\Dcp(U)$ denotes the respective space of distributions.
Finally, let $(u_n)$ be a fixed approximate unit in $\Dc(\er^d)=\Dc(E)$, i.e., $u_n(\x)=n^d\rho(n\x)$, where $\rho\in\Dc(\er^d)$ is a non-negative function with $\int_{\er^d}\rho=1$.

We will use also vector-valued functions. $\Dc(U,\er^d)$ will denote the space of $\C^\infty$-smooth vector fields with a compact support in $U$ and with values in $\er^d$, i.e., $d$-tuples of elements of $\Dc(U)$. If $\g = (g_1,\ldots,g_n)$ is a $d$-tuple of distributions, then $\div\g = \sum_{i=1}^n \partial_ig_i$.

By $L^1(U,E)$ we denote the space of all (equivalence classes of) integrable $E$-valued functions defined on $U$. 
The norm on this space is defined by
$$\norm{\f}=\int_{U} \norm{\f(\x)}_E\di\lambda^d(\x),\quad \f\in L^1(U,E).$$
Recall that $L^1(U,E)$ is canonically embedded to $L^1(\er^d,E)$,
where each $f\in L^1(U,E)$ is extended by zero outside $E$.

The space $L^\infty(U,E^*)$ is the space of all (equivalence classes of) essentially bounded measurable $E^*$-valued functions defined on $U$. The norm is defined by 
$$\norm{\g}=\esssup_{\x\in U} \norm{\g(\x)}_{E^*},\quad \g\in L^\infty(U,E^*).$$
The space $L^\infty(U,E^*)$ is canonically isometric to the dual of $L^1(U,E)$, the duality being defined by
$$\ip{\g}{\f} = \int_U \ip{\g(\x)}{\f(\x)}\di\lambda^d(\x), \quad \g\in L^\infty(U,E^*),\f\in L^1(U,E).$$

Finally, let $\Omega$ be a fixed nonempty convex open subset of $E$ with a base point $\o\in\Omega$. Without loss of generality we may and shall suppose that $\o$ is the origin.

\section{Proof of the main result}

The idea of the proof of our main result, Theorem \ref{t:main}, is to mimic the known and easy proof in dimension one. Let us recall it:

It is well known that the mapping $T:\Lip_0(\er)\ni f\mapsto f'\in L^\infty(\er)$ is an onto isometry. Consider the adjoint $T^*$. If $x>0$, then $T^*(\chi_{(0,x)})=\delta(x)$ since for any $F\in\Lip_0(\er)$ we have
$$\ip{T^*\chi_{(0,x)}}{F} =\ip{\chi_{(0,x)}}{TF}=\ip{\chi_{(0,x)}}{F'}=\int_0^x F'=F(x)=\ip{\delta(x)}{F}.$$
Similarly, $T^*(-\chi_{(x,0)})=\delta(x)$ for $x<0$. Since characteristic $\delta(\er)$ is linearly dense in $\F(\er)$ and characteristic functions of intervals are linearly dense in $L^1(\er)$, it easily follows that $T^*$ isometrically maps $L^1(\er)$ onto $\F(\er)$.

Hence, we try to mimic this approach in higher dimension and for general convex domains. We consider the mapping $T:\Lip_0(\Omega)\ni f\mapsto \nabla f\in L^\infty(\Omega,E^*)$. It is rather a standard fact that $T$ is an into isometry (cf. Proposition~\ref{P:2}(i)).
However, it is not onto, the range is described in Proposition \ref{P:2}(ii). It turns up that the range is weak$^*$-closed in $L^\infty(\Omega,E^*)$  and in Proposition \ref{P:3} we describe its pre-annihilator and hence, the respective quotient is a predual. Finally, in Proposition~\ref{P:main} we show, using the auxiliary Proposition~\ref{P:div}, that the adjoint map $T^*$ maps this predual isometrically onto $\F(\Omega)$ and give a representation of the mapping $\delta$. A large part of the first two steps, i.e., Proposition~\ref{P:2}(ii) 
and Proposition~\ref{P:3}, for the case $\Omega=\er^d$ are due to \cite{lerner}. 

We start by the following proposition. The parts (i) and (ii) are classical well-known facts, we provide references for the proof. The assertion (iii) is essentially standard as well, 
we provide the proof for the sake of completeness.

\begin{prop}\label{P:L1}

Let $F:\Omega\to\er$ be an $L$-Lipschitz function. Then the following hold:

\begin{itemize}

	\item[(i)] For almost all $\x\in\Omega$ there exists the Fr\'echet derivative $F'(\x)\in E^*$. Moreover, $\norm{F'(\x)}\le L$ whenever $F'(\x)$ exists.

	\item[(ii)] The mapping $F':\x\mapsto F'(\x)$ is an almost everywhere defined measurable function. Moreover, it is the gradient of $F$ in the sense of distributions on $\Omega$, i.e., 
	$$\int_\Omega F'\cdot \varphi \di\lambda^d=-\int_\Omega F\cdot \nabla\varphi\di\lambda^d,\quad \varphi\in\Dc(\Omega).$$

	\item[(iii)] For each $\x,\y\in\Omega$ we have
	$$F(\y)-F(\x)=\lim_{n\to\infty} \int_0^1 \ip{(F'*u_n)(\x+t(\y-\x))}{\y-\x}\di t.$$

\end{itemize}

\end{prop}

\begin{proof} (i) $F$ is differentiable almost everywhere by the classical Rademacher theorem (see, e.g., \cite[Theorem 3.1.6]{federer} or \cite[Theorem 30.3]{LM}). The estimate of the norm is obvious.

(ii) The set of Fr\'echet-differentiability points is of full measure, hence Lebesgue measurable, by (i) (in fact, it is a Borel set by \cite[Lemma 30.2]{LM}, for a more general result see \cite[Theorem 2]{zaj91}). Moreover, since Fr\'echet derivative is determined by partial derivatives and partial derivatives of a continuous function are clearly of the first Baire class, the mapping $F'$ is measurable. The integral formula  follows easily using Fubini theorem and integration by parts for absolutely continuous functions. 

(iii) Let us first show the formula in case $\Omega=\er^d$. Then for each $n\in\en$ the convolution $F*u_n$ is a $\C^\infty$ function on $\er^d$ and its gradient satisfies $\nabla(F*u_n)= F'*u_n$. Hence we have
$$(F*u_n)(\y)-(F*u_n)(\x)=\int_0^1 \ip{(F'*u_n)(\x+t(\y-\x))}{\y-\x}\di t,\ \x,\y\in\er^d,n\in\en.$$
Since $F*u_n\to F$ pointwise (in fact, uniformly on compact sets), we can take the limit to obtain the formula.

Next suppose that $\Omega\subsetneqq\er^d$. In this case the convolutions including functions defined on $\Omega$ are understood, as usually, in the sense that the respective functions are extended by zero outside $\Omega$. 

Let $\tilde F$ be a Lipschitz extension of $F$ defined on $\er^d$. Such an extension exists, for example, due to \cite[Theorem 30.5]{LM}. (In fact, one can preserve the Lipschitz constant, but it is not important at this point.) Hence, using the case $\Omega=\er^d$ the formula holds with $\tilde F$ in place of $F$, i.e., 
$$\tilde F(\y)-\tilde F(\x)=\lim_{n\to\infty} \int_0^1 \ip{(\tilde F'*u_n)(\x+t(\y-\x))}{\y-\x}\di t,\quad \x,\y\in\er^d.$$
Given $\x,\y\in\Omega$, the segment $[\x,\y]$ is a compact subset of $\Omega$ and hence we can find $\varepsilon>0$ with $\varepsilon<\dist([\x,\y],\er^d\setminus\Omega)$. Further, by the properties of the approximate unit, there is some $n_0\in\en$ such that $\spt u_n\subset U(\o,\varepsilon)$ for $n\ge n_0$.
Then for $n\ge n_0$ one has
$$ (F'*u_n)(\x+t(\y-\x))=(\tilde F'*u_n)(\x+t(\y-\x)) \mbox{ for }t\in[0,1].$$
Hence the formula follows.
\end{proof}

\begin{prop}\label{P:2} For any $F\in\Lip_0(\Omega)$ set $T(F)=F'$. Then the following hold.

\begin{itemize}

	\item[(i)] $T$ is a linear isometry of $\Lip_0(\Omega)$ into $L^\infty(\Omega,E^*)$.

	\item[(ii)] The range of $T$ is
	$$X(\Omega)=\{\f=(f_i)_{i=1}^d\in L^\infty(\Omega,E^*)\setsep \partial_i f_j=\partial_j f_i\mbox{ for }i,j=1,\dots,d\},$$
	where the derivatives are considered in the sense of distributions on $\Omega$.

	\item[(iii)] The inverse operator $T^{-1}\colon X(\Omega)\to \Lip_0(\Omega)$ is defined by
	$$T^{-1}(\f)(\x)= \lim_{n\to\infty}\int_0^1 \ip{(\f*u_n)(t\x)}{\x}\di t,\ \f\in X(\Omega),\x\in\Omega.$$
\end{itemize}

\end{prop}

\begin{proof} It follows from Proposition~\ref{P:L1} that $T$ is a linear operator from $\Lip_0(\Omega)$ to $L^\infty(\Omega,E^*)$ with $\norm{T}\le 1$. Further, the formula for the inverse mapping follows from 
Proposition~\ref{P:L1}(iii). It remains to identify the range and to show that $T$ is an isometry.

Let us start by proving that $T$ is an isometry. Fix $F\in \Lip_0(\Omega)$ and set 
$$L=\norm{F'}_{L^\infty(\Omega,E^*)} =\esssup_{x\in\Omega}\norm{F'(\x)}_{E^*}.$$
Then $\norm{F'*u_n}_{L^\infty(\Omega,E^*)}\le L$ for each $n\in\en$. Indeed, fix $n\in\en$, $\x\in\Omega$ and $\h\in E$ with $\norm{\h}\le 1$. Then
$$\begin{aligned}
\abs{\ip{(F'*u_n)(\x)}{\h}}
&=\abs{\sum_{i=1}^d (\partial_i F * u_n)(\x) h_i}
=\abs{\sum_{i=1}^d \int_{\er^d}\partial_i F(\y)  u_n(\x-\y)\di \y \cdot h_i}
\\&=\abs{ \int_{\er^d}\sum_{i=1}^d\partial_i F(\y)\cdot h_i\cdot u_n(\x-\y)\di \y }
=\abs{\int_{\er^d}\ip{F'(\y)}{\h}  u_n(\x-\y)\di \y }
\\&\le\int_{\er^d}\abs{\ip{F'(\y)}{\h}}  u_n(\x-\y)\di \y 
\le \int_{\er^d}\norm{F'(\y)}_{E^*}  u_n(\x-\y)\di \y
\le L.\end{aligned}$$
(During the computation we again used the convention that the functions defined on $\Omega$ are extended by zero
on $\er^d\setminus\Omega$.)
Hence, indeed, $\norm{(F'*u_n)(\x)}\le L$ for each $n\in\en$ and $\x\in\Omega$. Therefore, for each $n\in\en$ and $\x,\y\in\Omega$ we have
$$\begin{aligned}\abs{\int_0^1 \ip{(F'*u_n)(\x+t(\y-\x))}{\y-\x}\di t}&\le 
\int_0^1 \abs{\ip{(F'*u_n)(\x+t(\y-\x))}{\y-\x}}\di t
\\&\le \int_0^1 \norm{(F'*u_n)(\x+t(\y-\x))}_{E^*}\cdot\norm{\y-\x}_E\di t
\\&\le L\norm{\y-\x}_E,
\end{aligned}$$
hence by the formula in Proposition~\ref{P:L1}(iii) the function $F$ is $L$-Lipschitz.

It remains to prove the assertion (ii). One one hand, it is obvious that the range of $T$ is contained in $X(\Omega)$, since, in the sense of distributions on $\Omega$, always $\partial_i\partial j F=\partial_j\partial_i F$ for $F\in \Lip_0(\Omega)$ (in fact, for any distribution $F\in\Dcp(\Omega)$). 

Conversely, suppose that $\f\in X(\Omega)$. Consider $\f$ to be extended by $0$ on $\er^d\setminus\Omega$. 
For $i,j\in\{1,\dots,d\}$ set
$$U_{i,j}=\partial_j f_i\mbox{ in }\Dcp(\Omega)\mbox{\quad and\quad}V_{i,j}=\partial_j f_i\mbox{ in }\Dcp(\er^d).$$
Clearly $U_{i,j}=V_{i,j}|_{\Dc(\Omega)}$ and $U_{i,j}=U_{j,i}$ for $i,i\in\{1,\dots,d\}$. 

Suppose without loss of generality that $\spt u_n\subset U(\o,\frac1n)$ for each $n\in\en$.
For $n\in\en$ set $\Omega_n=\{\x\in\Omega\setsep \dist(\x,\er^d\setminus\Omega)>\frac1n\}$. Then $\Omega_n$ is a convex open set and, moreover, $\o\in\Omega_n$ for $n$ large enough. Consider functions $u_n*\f$ for $n\in\en$. Then $u_n*\f$ is a $\C^\infty$ mapping defined on $\er^d$. Moreover, 
$$\partial_j(u_n*f_i)(\x)=\partial_i(u_n*f_j)(\x),\quad i,j\in\{1,\dots,d\},\x\in \Omega_n,n\in\en.$$
Indeed, fix $n\in\en$, $\x\in\Omega_n$ and $i,j\in\{1,\dots,d\}$. Then
$$\partial_j(u_n*f_i)(\x)=(u_n*V_{i,j})(\x)=V_{i,j}(\y\mapsto u_n(\x-\y))=U_{i,j}(\y\mapsto u_n(\x-\y)),$$
since the support of $\y\mapsto u_n(\x-\y)$ equals 
$$\x-\spt u_n\subset U(\x,\tfrac1n)\subset\Omega.$$
Thus we can conclude by recalling $U_{i,j}=U_{j,i}$.

It follows that there is a $\C^\infty$-function $F_n\colon \Omega_n\to\er$ with $\nabla F_n=u_n*\f$ on $\Omega_n$. Without loss of generality we can suppose $F_n(\o)=0$ for each $n\in\en$.
Further, since $\norm{u_n*\f}_{L^\infty(\Omega,E^*)}\le \norm{\f}_{L^\infty(\Omega,E^*)}$ for each $n\in\en$, each
$F_n$ is $L$-Lipschitz, where $L=\norm{\f}_{L^\infty(\Omega,E^*)}$. Since the sequence $(F_n)$ is uniformly Lipschitz
and $F_n(\o)=0$ for each $n\in\en$, it is easy to check it is locally uniformly bounded on $\Omega$. Therefore, by Arzel\`a-Ascoli theorem one can find a subsequence which converges to an $L$-Lipschitz function $F$ uniformly on compact subsets of $\Omega$. (Given $K\subset \Omega$ compact, then $K\subset\Omega_n$ for large $n$ and hence there is $n_0$ such that
$(F_n)_{n\ge n_0}$ is uniformly bounded and uniformly Lipschitz on $K$, so there is a subsequence uniformly convergent on $K$. A diagonal argument completes the construction.) Without loss of generality suppose $F_n\to F$. Then for any $\varphi\in\Dc(\Omega)$ we have
$$\int_{\Omega} F\cdot \nabla\varphi 
=\lim_{n\to\infty} \int_{\Omega_n} F_n\cdot \nabla\varphi 
=-\lim_{n\to\infty} \int_{\Omega_n} (u_n*\f)\cdot \varphi
=-\int_{\Omega} \f\cdot \varphi.$$
where we used that $\spt\varphi\subset\Omega_n$ for $n$ large enough. Hence $F'=\f$.
\end{proof}

\begin{remark}
The assertion (ii) of the previous proposition for the case $\Omega=\er^d$ is the content of \cite[Lemma 2]{lerner}. The proof of the nontrivial inclusion is done in a different way. First, if $\f\in X(\er^d)$,
by \cite[Proposition 4.3.9]{horvath} there is a distribution $U\in\Dcp(\er^d)$ such that $\f$ is the distributional derivative of $U$. 

This means that $U$ belongs to the space $L^1_\infty(\er^d)$ defined in \cite[Section~1.1]{Mazya}. It is shown in \cite[Theorem~1.1.2]{Mazya} that then $U\in W^{1,p}_{loc}(\er^d)$ for any $p>1$. 
Let $K=\sqrt d\int_0^1 t^{-d/p}\di t$. Lemma~4.28 in \cite{Adams-Fournier} gives 
\begin{equation}\label{eq:embedd}
\forall \x,\y\in\er^d,p>d, u\in C^\infty(\er^d): |u(\x)-u(\y)|\leq 2K\|\x-\y\|^{1-\frac dp}\|\nabla u\|_{L^p(Q)},
\end{equation}
where $Q$ is a cube with sidelength $2\|\x-\y\|$ that contains $\x$ and $\y$. 
From \cite[Theorem~4.12, Part I]{Adams-Fournier} the continuous embedding of $W^{1,p}_{loc}(\er^d)$ into $C(\er^d)$ follows. Consequently, \eqref{eq:embedd} holds for all $u\in W^{1,p}_{loc}(\er^d)$ whenever $p>d$. 
Application of \eqref{eq:embedd} to $U$ together with H\"older's inequality and the fact that $\nabla U\in L^\infty(\er^d)$ gives 
$$
\forall \x,\y\in\er^d,p>d: |U(\x)-U(\y)|\leq 2^{1+\frac 1p}K\|\x-\y\|^{1-\frac dp}\|\nabla U\|_{L^p(Q)}\leq 2^{1+\frac 1p}K\|\x-\y\|\|\f\|_\infty ,
$$
i.e., $U$ is a Lipschitz function.

A similar argument probably could be done for any convex domain, but it should be more technical and
we have not find any explicit reference for it. Therefore we give above a direct proof of the assertion (ii) for a general convex domain. 

We further remark that $X(\Omega)$ is a known object, the elements of $X(\Omega)$ are called \emph{closed $L^\infty(\Omega)$ currents}.
\end{remark}

\begin{prop}\label{P:3} Set
$$Y(\Omega)=\{\g\in L^1(\Omega,E)\setsep \div \g=0\mbox{ in }\Dcp(\er^d)\}.$$
Then the following hold:
\begin{itemize}

	\item $X(\Omega)=Y(\Omega)^\perp$ and $Y(\Omega)=(X(\Omega))_\perp$ in the standard duality
$(L^1(\Omega,E))^*=L^\infty(\Omega,E^*)$.

  \item $Y(\Omega)\cap\Dc(\Omega,E)$ is dense in $Y(\Omega)$.
\end{itemize}
\end{prop}

\begin{proof} Let us recall any $\g\in L^1(\Omega,E)$ is canonically identified with an element of $L^1(\er^d,E)$  ($\g$ is extended by zero outside $\Omega$). Hence, if $\g\in L^1(\Omega,E)$, then $\div \g$ in $\Dcp(\er^d)$ is the distributional divergence of the described extension of $\g$.

Let us start the proof by showing $X(\Omega)=(Y(\Omega)\cap \Dc(\Omega,E))^\perp$:
\begin{itemize}

	\item[$\supset$:] Fix $\f\in (Y(\Omega)\cap \Dc(\Omega,E))^\perp$. To prove that $\f\in X(\Omega)$, take 
arbitrary $i,j\in\{1,\dots,d\}$ distinct and $\varphi\in\Dc(\Omega)$. Set $g_i=\partial_j\varphi$, $g_j=-\partial_i\varphi$ and $g_k=0$ for $k\in\{1,\dots,d\}\setminus\{i,j\}$. Then
$$\ip{\partial_i f_j-\partial_j f_i}{\varphi}=-\ip{f_j}{\partial_i\varphi}+\ip{f_i}{\partial_j\varphi}=\ip{\f}{\g}=0.$$

\item[$\subset$:]  Suppose that $\f\in X(\Omega)$ and take $\g\in Y(\Omega)\cap\Dc(\Omega,E)$. We will show that $\ip{\f}{\g}=0$. By Proposition~\ref{P:2} we know that there is $F\in\Lip_0(\Omega)$ with $F'=\f$. Then $$\ip{\f}{\g}=\ip{F'}{\g}=-\ip{F}{\div\g}=0.$$
\end{itemize}

Hence, by the Hahn-Banach theorem we have $(X(\Omega))_\perp=\overline{Y(\Omega)\cap\Dc(\Omega,E)}$. Therefore, to complete the proof it 
is enough to show that $Y(\Omega)\subset(X(\Omega))_\perp$.

Let us prove it first in case $\Omega=E$. Take $\f\in X(\Omega)$ and $\g\in Y(\Omega)$. 

In the first step suppose that $\g\in\C^\infty(\er^d)$. Fix some $\psi\in\Dc(\er^d)$ with $\psi=1$ on $U(\o,1)$, $\spt\psi\subset U(\o,2)$ and $0\le\psi\le 1$ on $\er^d$. (The balls are taken with respect to the norm of $E$.) For $n\in\en$ set $\g_n(\x)=\psi(\frac{\x}{n})\cdot\g(\x)$.
Let $F\in\Lip_0(\er^d)$ be such that $\f=F'$. Denote by $L$ the Lipschitz constant of $F$ (i.e., $L=\norm{\f}_{L^\infty(\er^d,E^*)}$). Then 
$$\begin{aligned}
\abs{\ip{\f}{\g_n}}& = \abs{\ip{F'}{\g_n}}=\abs{\ip{F}{\div\g_n}}
\\&=\abs{\int_{\er^d}F(\x)\cdot\left(\psi(\tfrac{\x}{n})\div\g(\x) + \tfrac1n\ip{\g(\x)}{\nabla\psi(\tfrac{\x}{n})}\right)\di\x}
\\&=\abs{\int_{\er^d}F(\x)\cdot \frac1n\ip{\g(\x)}{\nabla\psi(\frac{\x}{n})}\di\x}
\\&=\abs{\int_{U(\o,2n)\setminus U(\o,n)}F(\x)\cdot \tfrac1n\ip{\g(\x)}{\nabla\psi(\tfrac{\x}{n})}\di\x}
\\&\le\frac1n\cdot\sup_{\x\in U(\o,2n)}\cdot\abs{F(\x)}\cdot\norm{\nabla\psi}_{L^\infty(\er^d,E^*)}\cdot
\int_{\er^d\setminus U(\o,n)} \norm{\g(\x)}_E\di x
\\&\le 2L\cdot\norm{\nabla\psi}_{L^\infty(\er^d,E^*)}\cdot
\int_{\er^d\setminus U(\o,n)} \norm{\g(\x)}_E\di x\to 0\end{aligned}$$
for $n\to\infty$. Since $\g_n\to\g$ in $L^1(\er^d,E)$, we conclude that $\ip{\f}{\g}=0$.

In the second step let $\g\in Y^d$ be arbitrary. Then for each $n\in\en$ we have $u_n*\g\in\C^\infty(\er^d,E)$,
$\div(u_n*\g)=u_n*\div\g=0$, hence $\ip{\f}{u_n*\g}=0$. Since $u_n*\g\to\g$ in $L^1(\er^d,E)$, necessarily $\ip{\f}{\g}=0$.

Finally, let $\Omega$ be arbitrary, $\f\in X(\Omega)$ and $\g\in Y(\Omega)$. Let  $F\in\Lip_0(\Omega)$ be such that $\f=F'$.
Let $\tilde F\in\Lip_0(E)$ be an extension of $F$. Then, using the case $\Omega=E$ and the assumption $\div \g=0$ in $\Dcp(\er^d)$ (and not just in $\Dcp(\Omega)$), we have
$$\begin{aligned}\ip{\f}{\g}&=\int_{\Omega}\ip{F'(\x)}{\g(\x)}\di\x 
=\int_{\Omega}\ip{\tilde F'(\x)}{\g(\x)}\di\x 
\\&=\int_{\er^d}\ip{\tilde F'(\x)}{\g(\x)}\di\x 
=\ip{\tilde F'}{\g}=0.\end{aligned}$$
\end{proof}

\begin{remark}
(1) By Proposition \ref{P:2}  $Lip_0(\Omega)$ is isometric with $X(\Omega)$, by Proposition \ref{P:3} $X(\Omega)$ is isometric with $(L^1(\Omega,E)/_{Y^d})^*$. Hence  $Lip_0(\Omega)$ is isometric with
$(L^1(\Omega,E)/_{Y^d})^*$. 
For the case $\Omega=\er^d$ this is the content of \cite[Theorem 3]{lerner}.

(2) In \cite{lerner} the proof of the case $\Omega=\er^d$ of  Proposition~\ref{P:3} is contained in Lemmata 4 and 5, although the equality $X(\er^d)=Y(\er^d)^\perp$ is not explicitly mentioned there.
\end{remark}

\begin{prop}\label{P:div} Let $\a\in\Omega$ be fixed. Then there is $\g\in L^1(\Omega,\er^d)$ with compact support in $\Omega$
such that $\div\g=\ep_\o-\ep_\a$ in $\Dcp(\er^d)$.
\end{prop}

\begin{proof} In this proof we will consider $\er^d$ with the euclidean norm.
Set 
$$\h(\x)=\frac{\x}{d\kappa_d\norm{\x}^d},\quad \x\in\er^d\setminus\{\o\},$$
where $\kappa_d$ is the volume of the $d$-dimensional ball. Then $\h$ is the gradient of the standard
fundamental solution of the Laplace equation in $\er^d$. In particular,
\begin{gather}
\div\h=\ep_\o \mbox{ in }\Dcp(\er^d),\\
\h\in\C^\infty(\er^d\setminus\{\o\}), \\
\div\h(\x)=0\mbox{ for } \x\in\er^d\setminus\{\o\}. \label{eq:3}
\end{gather}
Further, for any $r>0$ we have
\begin{equation}
\label{eq:4}\int_{\partial U(\o,r)} \ip{\h(\x)}{\nu(\x)}\di\H^{n-1}(\x) =1,
\end{equation}
where $\nu(\x)$ denotes the outer normal at the point $\x$ and $U(\o,r)$ is the Euclidean ball centered at $\o$ with radius $r$. Indeed, since $\nu(\x)=\frac{\x}{\norm{\x}}$, one
has
$$\begin{aligned}\int_{\partial U(\o,r)} \ip{\h(\x)}{\nu(\x)}\di\H^{n-1}(\x)&=\int_{\partial U(\o,r)} \frac{1}{d \kappa_d\norm{\x}^{d-1}}\di\H^{n-1}(\x)\\&=\frac{\H^{n-1}(\partial U(\o,r))}{d\kappa_d r^{d-1}}
=\frac{\H^{n-1}(\partial U(\o,1))}{d\kappa_d}=1.\end{aligned}$$
The equation \eqref{eq:4} can be extended to more general domains:
\begin{equation}
\label{eq:5}\begin{aligned} \mbox{If $U$ is a bounded convex domain containing $\o$,}&\\ \mbox{then }\int_{\partial U} \ip{\h(\x)}{\nu(\x)}&\di\H^{n-1}(\x) =1.\end{aligned}
\end{equation}
Indeed, let $r>0$ be so small that $\overline{U(\o,r)}\subset U$. Then, by \eqref{eq:3} and the Gauss theorem \cite[Theorem 37.22]{LM} (note that the boundary of a convex domain is Lipschitz) we have
$$0=\int_{U\setminus\overline{U(\o,r)}} \div\h=\int_{\partial U} \ip{\h(\x)}{\nu(\x)}\di\H^{n-1}(\x)-\int_{\partial U(\o,r)} \ip{\h(\x)}{\nu(\x)}\di\H^{n-1}(\x),$$
thus \eqref{eq:5} follows from \eqref{eq:4}.

We continue by setting
$$\h_\a(\x)=\h(\x)-\h(\x-\a),\quad \x\in\er^d\setminus\{\a,\o\}.$$
Then clearly
\begin{gather}
\div\h_\a=\ep_\o-\ep_\a \mbox{ in }\Dcp(\er^d),\label{eq:div}\\
\h_\a\in\C^\infty(\er^d\setminus\{\o,\a\}), \\
\div\h_\a(\x)=0\mbox{ for } \x\in\er^d\setminus\{\o,\a\}. \label{eq:8}
\end{gather}
and, moreover, by \eqref{eq:5} we get
\begin{equation}\label{eq:6} 
\begin{aligned}\mbox{If $U$ is a bounded convex domain containing $\{\o,\a\}$, then }\\\int_{\partial U} \ip{\h_\a(\x)}{\nu(\x)}\di\H^{n-1}(\x) =0.\end{aligned}
\end{equation}

Since the segment $[\o,\a]$ is a compact subset of $\Omega$, there is $r>0$ such that
$$\overline{[\o,\a]+U(\o,r)}\subset \Omega.$$
Choose some $\eta\in\Dc(\er^d)$ such that
$$\eta=1 \mbox{ on } [\o,\a]+U(0,\tfrac r2) \mbox{ and }\spt\eta\subset[\o,\a]+U(0,\tfrac34r).$$
Then we have
\begin{itemize}
	\item $\eta\h_\a\in L^1(\Omega,\er^d)$ and has a compact support in $\Omega$;
	\item $\eta\h_\a\in\C^\infty(\er^d\setminus\{\o,\a\})$;
	\item $\div \eta\h_\a(\x)=0$ for $\x\in ([\o,\a]+U(0,\tfrac r2))\setminus\{\o,\a\}$;
	\item $\div \eta\h_\a(\x)=0$ for $\x\in \er^d\setminus([\o,\a]+U(0,\frac34r))$.
\end{itemize}
 Hence, if we set $U_1=[\o,\a]+U(0,\frac14r)$ and $U_2=[\o,\a]+U(0,r)$, we get
\begin{multline*}\int_{U_2\setminus\overline{U_1}}\div(\eta\h_\a)\di\lambda^d= \int_{\partial U_2} \ip{\eta(\x)\h_\a(\x)}{\nu(\x)}\di\H^{n-1}(\x) \\\qquad \qquad 
-\int_{\partial U_1} \ip{\eta(\x)\h_\a(\x)}{\nu(\x)}\di\H^{n-1}(\x)
=0-\int_{\partial U_1} \ip{\h_\a(\x)}{\nu(\x)}\di\H^{n-1}(\x)=0.\end{multline*}
Since $\div(\eta\h_\a)$ restricted to $U_2\setminus \overline{U_1}$ is a $\C^\infty$-function with a compact support,
by \cite[Auxiliary lemma 3.15]{NoSt} there is $\g_0\in\Dc(U_2\setminus \overline{U_1},\er^d)$ with $\div\g_0(x)=\div(\eta\h_\a)(x)$ for $x\in U_2\setminus \overline{U_1}$. Set $\g=\eta\h_\a-\g_0$. Then $\g\in L^1(\Omega,\er^d)$ with compact support in $\Omega$. Moreover, we will show that $\div\g=\varepsilon_\o-\varepsilon_\a$ in $\Dc(\er^d)$. Due to \eqref{eq:div} it is enough to show that $\div(\g-\h_\a)=0$ in $\Dc(\er^d)$. But 
$$\g-\h_\a=(\eta-1)\h_\a-\g_0$$
is a $\C^\infty$ vector field on $\er^d$ (note that $\eta-1=0$
on $[\o,\a]+U(\o,\frac{r}{2})$), hence its distributional divergence
is a $\C^\infty$ function and can be computed pointwise.
Let us consider the following three open sets covering $\er^d$:

(i) On $\overline{U_1}$ we have $\eta=1$ and $\g_0=0$, hence
$(\eta-1)\h_\a-\g_0=0$. Thus 
$\div((\eta-1)\h_\a-\g_0)(\x)=0$ for each $\x\in[\o,\a]+U(\o,\frac{r}{2})$.

(ii) On $\er^d\setminus U_2$ we have $\eta=0$ and $\g_0=0$, hence $(\eta-1)\h_\a-\g_0=-\h_\a$, therefore the divergence is zero at each point of this set by \eqref{eq:8}.

(iii) For $x\in U_2\setminus\overline{U_1}$ we have
$$\div((\eta-1)h_\a-\g_0)(\x)=\div(\eta\h_a)(\x)-\div h_\a(\x)-\div\g_0(\x)=0$$
by \eqref{eq:8} and the choice of $\g_0$.

This completes the proof.
\end{proof}

The final ingredient is the following proposition, the proof of Theorem~\ref{t:main} then immediately follows.

\begin{prop}\label{P:main} Let $T:F\mapsto F'$ be the isometry from Proposition~\ref{P:2}. Then the dual mapping $T^*$ maps $(L^1(\Omega,E)/Y(\Omega))$ onto
$\F(\Omega)$. Moreover, for $\g\in L^1(\Omega,E)$ one has $T^*([\g])=\delta(\a)$ if and only if $\div \g=\ep_\o-\ep_\a$ in $\Dcp(\er^d)$.
\end{prop}

\begin{proof} Fix $\a\in\Omega$. Let $\g\in L^1(\Omega,E)$ be a mapping with compact support and satisfying $\div\g=\ep_\o-\ep_\a$ in $\Dcp(\er^d)$. (Such a mapping exists by Proposition~\ref{P:div}.) For any $F\in\Lip_0(\Omega)$ we then have
$$\begin{aligned}
\ip{T^*([\g])}{F}&=\ip{[\g]}{TF}=\ip{[\g]}{F'}
=\int_\Omega \ip{\g(\x)}{F'(\x)}\di\x
\\&=\lim_{n\to\infty} \int_\Omega \ip{(u_n*\g)(\x)}{F'(\x)}\di\x
=\lim_{n\to\infty} \ip{(u_n*\g)}{F'}
\\&	=-\lim_{n\to\infty} \ip{\div (u_n*\g)}{F}
=-\lim_{n\to\infty} \int_\Omega F(\x)\cdot \div (u_n*\g)(\x)\di\x
\\&=-\lim_{n\to\infty} \int_\Omega F(\x)\cdot (u_n*\div\g)(\x)\di\x
\\&=-\lim_{n\to\infty} \int_\Omega F(\x)\cdot (u_n*(\ep_\o-\ep_\a))(\x)\di\x
\\&=-\lim_{n\to\infty}\int_\Omega  F(\x)\cdot (u_n(\x)-u_n(\x-\a))\di\x
\\&= \lim_{n\to\infty} ((F*\widecheck{u_n})(\a)-(F*\widecheck{u_n})(\o))
= F(\a)-F(\o)=\ip{\delta(\a)}{F}.
\end{aligned}$$
The first three equalities are just applications of definitions. The fourth one follows from the fact that $u_n*\g\to\g$ in the $L^1$-norm. Since $\g$ has compact support in $\Omega$, $u_n*\g\in\Dc(\Omega,\er^d)$ for large $n$ and hence the fifth one is just rewriting the expression in the sense of distributions and the sixth one follows from the differentiation rules for distributions. Then a standard calculation follows. Let us point out that by $\widecheck{u_n}$ we mean the function defined by a formula $\widecheck{u_n}(\x)=u_n(-\x)$, $\x\in\er^d$, and that we use the obvious fact that the sequence $(\widecheck{u_n})$ is also an approximate unit.

Hence, we conclude that $T^*([\g])=\delta(\a)$.
Now, let $\g_1\in L^1(\Omega,E)$ be arbitrary. Since $T^*$ is an isometry, $T^*([\g_1])=\delta(\a)$ if and only if $[\g_1]=[\g]$, i.e. $\div(\g_1-\g)=0$ in $\Dcp(\er^d)$,
or, equivalently, $\div\g_1=\delta_\o-\delta_\a$ in $\Dcp(\er^d)$. 

To conclude the proof, observe that $(T^*)^{-1}(\delta(\a))\in (L^1(\Omega,E)/Y(\Omega))$ for any $\a\in\Omega$, hence $(T^*)^{-1}$ maps $\F(\Omega)$ into $(L^1(\Omega,E)/Y(\Omega))$. Let 
$S$ be the restriction of $(T^*)^{-1}$ to $\F(\Omega)$. Since $S^*=T^{-1}$ is an onto isometry,
necessarily $S$ is also an onto isometry. This completes the proof.
\end{proof}

\section*{Acknowledgement}

We are grateful to Professor N.~Lerner for sending us his preprint \cite{lerner}.

%%%%%%%%%%%%%%%%%%%%%%%%%%%%%%%%%%%%%%%%%%%%%%%%%%%%

\def\cprime{$'$}

%\bibliography{lip2}\bibliographystyle{acm}

\end{document}